\documentclass[12pt]{amsart}
\usepackage{amssymb}
\usepackage{verbatim}
\usepackage[usenames]{color}
\usepackage{hyperref}
\usepackage{url}
\usepackage[all]{xy}

\newtheorem{theorem}{Theorem}[section]
\newtheorem{proposition}[theorem]{Proposition} 
\newtheorem{lemma}[theorem]{Lemma}

\theoremstyle{definition}
\newtheorem{definition}[theorem]{Definition}

\theoremstyle{remark}

\newcommand{\defemph}{\textit}

\newcommand{\m}{-}

\newcommand{\mc}{\mathcal}
\newcommand{\mbb}{\mathbb}

\newcommand{\forces}{\Vdash}

\newcommand{\baire}{{^\omega}\omega}
\newcommand{\bairenodes}{{^{<\omega}\omega}}

\title[No Unwanted Universally Baire Morphisms]
 {No Unwanted \\ Universally Baire Morphisms}
\author{Dan Hathaway}
\address{Mathematics Department\\
University of Michigan\\
Ann Arbor, MI 48109--1043, U.S.A.}
\email{danhath@umich.edu}

\begin{document}

\begin{abstract}
We show that the usual proof
 that there are no morphisms,
 whose constituent maps are Borel,
 between certain challenge-response relations
 generalizes to show that there
 are no morphisms whose constituent
 maps are universally Baire.
\end{abstract}

\maketitle

\section{Morphisms}

Let $\kappa$ be a cardinal.
Recall that a set
 $A \subseteq \baire$
 is $\kappa$-\textit{universally Baire}
(see \cite{Feng})
 iff there exist trees
 $T, S \subseteq
  {^{<\omega} \omega} \times
  {^{<\omega} \delta}$
 for some cardinal $\delta$
 such that $p[T] = A$
 and in every forcing extension of $V$
 by a forcing of size $\le \kappa$,
 $p[T] = \baire \m p[S]$.
A set $A \subseteq \baire$
 is \textit{universally Baire}
 iff it is $\kappa$-universally Baire
 for all $\kappa$.
We make a similar definition for relations
 on $\baire$ to be universally Baire.
We say that a function is universally Baire
 iff its graph is.
Given a set $A$ which is $\kappa$-universally Baire,
 witnessed by $T$ and $S$,
 and given a forcing of size $\le \kappa$,
 we say that the set $p[T]$ (as computed in the extension)
 is what $A$ \textit{lifts} to.

Let $\mc{R}_1 := \langle \baire, \baire, R_1 \rangle$
 and $\mc{R}_2 := \langle \baire, \baire, R_2 \rangle$
 be challenge-response relations
 (so $R_1, R_2 \subseteq \baire \times \baire$).
It is natural to ask if there is a \textit{morphism}
 form the first to the second.
That is, a pair $\langle \phi_-, \phi_+ \rangle$
 of functions $\phi_-, \phi_+ : \baire \to \baire$
 such that
 $$(\forall x \in \baire)(\forall y \in \baire)\,
 \phi_-(x) R_1 y \Rightarrow x R_2 \phi_+(y).$$
In \cite{Blass} (Theorem 4.15),
 a situation is given where there
 can be no such morphism
 with \textit{either} $\phi_-$ or $\phi_+$ Borel.
The goal of this document is to
 show why in the same situation
 it is impossible for
 \textit{both} $\phi_-$ and $\phi_+$
 to be universally Baire.
This leaves open the question
 of whether one of $\phi_-$ or $\phi_+$
 could be universally Baire..

Consider a challenge-response relation
 $\mc{R} = \langle \baire, \baire, R \rangle$
 which has an interpretation in every
 forcing extension (this happens when
 the relation $R$ is universally Baire for example).
Given a forcing $\mbb{P}$,
 we say that that $\mbb{P}$ is
 $\mc{R}$-\textit{adequate} if
 $$1_\mbb{P} \forces
 (\forall x \in \baire)
 (\exists y \in \baire \cap \check{V})\,
 x R y.$$
\begin{lemma}
Let $\kappa$ be an infinite cardinal.
Let $f : \baire \to \baire$ be a function
 whose graph is $\kappa$-universally Baire.
Then in every forcing extension
 by a poset of size $\le \kappa$,
 $f$ lifts to a function
 defined on all of $\baire$.
\end{lemma}
\begin{proof}
Let $\kappa$ be a cardinal
 and let $T, S$ be trees witnessing
 that the graph of $f$ is
 $\kappa$-universally Baire.
Let $\mbb{P}$ be a poset of size
 $\le \kappa$.
Let $G$ be $(V,\mbb{P})$-generic.
We want to show that $(p[T])^{V[G]}$
 is the graph of a total function
 in $V[G]$.
Given any
 $x \in (\baire)^{V[G]}$,
 we want some $y \in (\baire)^{V[G]}$
 satisfying $(x,y) \in p[T]$.

Towards a contradiction,
 fix an $x \in (\baire)^{V[G]}$
 such that there is no
 such corresponding $y$.
From $T \subseteq {^{<\omega}\omega
 \times {^{<\omega}\omega}
 \times {^{<\omega}\delta}}$ we can form
 $T_x \subseteq { ^{<\omega}\omega
 \times {^{<\omega}\delta}}$
 which is a well-founded tree.
Since $T_x$ is well-founded,
 it has some rank function
 $\sigma : T_x \to \omega_1$.
Consider the tree
 $W$ whose nodes are pairs
 consisting of an element of
 $\bairenodes$
 and a partial attempt to build
 a rank function for the corresponding
 tree from $T$.
We have that
 $(x,\sigma)$ is a path through $W$.
Since $W \in V$ and $W$ has
 a path in $V[G]$,
 it has a path in $V$.
Such a path witnesses that
 $f$ is not total in $V$,
 which is a contradiction.

Using similar reasoning,
 it can be shown that in $V[G]$
 there are not $x,y_1,y_2$ with
 $y_1 \not\in y_2$ such that
 $(x,y_1) \in p[T]$ and
 $(x,y_2) \in p[T]$.
\end{proof}
We will show the following.
The hypothesis arises in practice,
 for example the proof that there
 is no Borel morphism from the
 splitting relation to the domination
 relation
 (see \cite{Blass} Theorem 4.15).
\begin{proposition}
\label{main_prop}
Let $\mc{R}_1 = \langle \baire, \baire, R_1 \rangle$
 and $\mc{R}_2 = \langle \baire, \baire, R_2 \rangle$
 be challenge-response relations with
 $R_1$ and $R_2$ universally Baire.
Suppose there is a forcing $\mbb{P}$
 which is $\mc{R}_1$-adequate
 but not $\mc{R}_2$-adequate.
Then there is no morphism
 $\langle \phi_-, \phi_+ \rangle$
 from $\mc{R}_1$ to $\mc{R}_2$
 such that both the graph of
 $\phi_-$ and the graph of $\phi_+$
 are universally Baire.
\end{proposition}
\begin{proof}
Consider a pair of functions
 $\langle \phi_-, \phi_+ \rangle$
 that are universally Baire.
Let $\Phi$ be the statement that
 there exist $x_1, x_2, y_1, y_2 \in \baire$
 satisfying the following
\begin{itemize}
\item[1)] $\phi_-(x_1) = x_2$;
\item[2)] $\phi_+(y_1) = y_2$;
\item[3)] $x_2 R_1 y_1$;
\item[4)] $\neg x_1 R_2 y_2$.
\end{itemize}
$$
\xymatrix{
 x_2 & R_1 & y_1 \ar[d]^{\phi_+} \\
 x_1 \ar[u]^{\phi_-} & \neg R_2 & y_2.
 }
$$
We claim that $\Phi$ is equivalent
 to a statement which asserts the existence
 of a path through a tree
 in the ground model.
The conditions 1)-4) are all of this form.
For example,
 if $R_2 = p[T] = \baire \m p[S]$,
 then 4) is equivalent to saying that
 $(x_1, y_2)$ is a path through $S$.

Note that by the definition of a morphism,
 if $\Phi$ holds, then
 $\langle \phi_-, \phi_+ \rangle$
 is not a morphism from $\mc{R}_1$ to $\mc{R}_2$.
Now since $\Phi$ is equivalent to
 a statement which asserts the existence of a
 path through a tree in the ground model,
 it is absolute between $V$ and forcing extensions.
In particular, if we show that $\Phi$ holds
 after forcing with $\mbb{P}$, then we are done.

Force with $\mbb{P}$ to get $V[G]$.
Let $x_1 \in (\baire)^{V[G]}$ witness that
 $\mbb{P}$ is not $\mc{R}_2$-adequate.
That is, there is no $y \in \baire \cap V$
 satisfying $x_1 R_2 y$.
By the lemma above,
 we may speak of $\phi_-(x_1)$.
Since $\mbb{P}$ is $\mc{R}_1$-adequate,
 let $y_1 \in \baire \cap V$ satisfy
 $\phi_-(x_1) R_1 y_1$.
By what we said about $x_1$, we have
 $\neg x_1 R_2 \phi_+(y_1)$.
Hence, $\Phi$ is satisfied.
This completes the proof.
\end{proof}

Note that this
 proposition says that
 $\phi_-$ and $\phi_+$ cannot
 \textit{both} be universally Baire,
 whereas Theorem 4.15 of \cite{Blass}
 says that \textit{neither}
 $\phi_-$ nor $\phi_+$ can be Borel.

\section{Weak Morphisms}

There is a variant of the notion of morphism
 which is more general when one does not
 assume the Axiom of Choice.
The idea is to replace functions
 with multiple valued functions.

\begin{definition}
Given challenge-response relations
 $\mc{A} = \langle \baire, \baire, A \rangle$ and
 $\mc{B} = \langle \baire, \baire, B \rangle$,
 a \defemph{weak morphism}
 from $\mc{A}$ to $\mc{B}$ is a pair
 $\langle F_-, F_+ \rangle$ of relations
 $F_-, F_+ \subseteq \baire \times \baire$
 such that the following are satisfied:
\begin{itemize}
\item[1)] $(\forall b_- \in \baire)(\exists a_- \in \baire)\,
 b_- F_- a_-$;
\item[2)] $(\forall a_+ \in \baire)(\exists b_+ \in \baire)\,
 a_+ F_+ b_+$;
\item[3)] for all $a_-, a_+, b_-, b_+ \in \baire$,
 if $b_- F_- a_-$, $a_+ F_+ b_+$, and $a_- A a_+$,
 then $b_- B b_+$.
\end{itemize}
\end{definition}
This definition can be remembered by the following picture.
$$
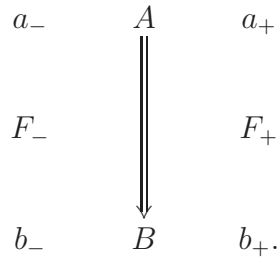
\xymatrix{
 a_- & A \ar@{=>}[dd] & a_+ \\
 F_-    &   & F_+ \\
 b_- & B & b_+.
}$$
If there is a morphism from
 $\mc{A}$ to $\mc{B}$,
 then there is a weak morphism from
 $\mc{A}$ to $\mc{B}$.
Assuming the Axiom of Choice,
 the other direction holds as well.
The proof of Proposition~\ref{main_prop}
 can be easily modified to prove
 the corresponding result
 for weak morphisms.

\end{document}